\documentclass[12pt,a4paper]{article}
\usepackage[utf8]{inputenc}
\usepackage{amsmath}
\usepackage{amsfonts}
\usepackage{amssymb}
\usepackage{amsthm}
\usepackage{graphicx}

\newtheorem{theorem}{Theorem}
\newtheorem{definition}{Definition}
\newtheorem{remark}{Remark}

\title{Polynomials invertible in $k$-radicals}
\author{Y. Burda, A. Khovanskii}

\begin{document}

\maketitle

\begin{abstract}
A classic result of Ritt describes polynomials invertible in radicals: they are compositions of power polynomials, Chebyshev polynomials and polynomials of degree at most 4. In this paper we prove that a polynomial invertible in radicals and solutions of equations of degree at most $k$ is a composition of power polynomials, Chebyshev polynomials, polynomials of degree at most $k$ and, if $k\leq 14$, certain polynomials with exceptional monodromy groups. A description of these exceptional polynomials is given. The proofs rely on classification of monodromy groups of primitive polynomials  obtained by M\"{u}ller based on group-theoretical results of Feit and on previous work on primitive polynomials with exceptional monodromy groups by many authors.
\end{abstract}

\section{Introduction}

This paper is devoted to a generalization of a result of Ritt on polynomials invertible in radicals:

\begin{theorem}[Ritt, \cite{Ritt22Radicals}]
\label{thm:ritt}
The inverse function of a polynomial with complex coefficients can be represented by radicals if and only if the polynomial is a composition of linear polynomials, the power polynomials $z\to z^n$, Chebyshev polynomials and polynomials of degree 4.
\end{theorem}

In the paper we give a complete description of polynomials invertible in $k$-radicals, i.e. in radicals and solutions of equations of degree at most $k$. The main result appears in Theorem \ref{thm:ksolvabilitynonprimitive}. A more complete description of polynomials appearing in its formulation appears in sections \ref{section:deg5}-\ref{section:highdegree}.

The description of polynomials invertible in $k$-radicals uses deep group-theoretical result of Feit on primitive permutation groups containing a full cycle, its refinement obtained by G. Jones and work of P. M\"{u}ller that builds on it to provide a classification of monodromy groups of primitive polynomials.

Description of the polynomials with primitive monodromy groups that appear in formulation of Theorem \ref{thm:ksolvabilitynonprimitive} is mostly known. However it is scattered among many papers, is not complete and is not formulated in a way that we needed for our purposes. To get a description that suited our purposes we consulted P. M\"{u}ller. We are especially thankful to Alexandr Zvonkin who helped us understand many of the results. However we still have some unanswered questions on exceptional polynomials of degree 15. The corresponding result had been obtained in \cite{CassouNogues99Factorizations}, however it is formulated there with too few details and without a proof. We provide a tentative description of these polynomials in \ref{section:deg15}. We will improve it before submitting a final version of this paper to print.

We would like to thank P. M\"{u}ller for his answers to our questions. We are especially thankful to Alexandr Zvonkin, whose generous help had been of great use to us.

\section{Formulation of the problem and its answer}

\begin{definition}
Let $k$ be a natural number. A field extension $L/K$ is $k$-radical if there exists a tower of extensions $K=K_0\subset K_1\subset \ldots \subset K_n$ such that $L\subset K_n$ and for each $i$, $K_{i+1}$ is obtained from $K_i$ by adjoining an element $a_i$, which is either a solution of an algebraic equation of degree at most $k$ over $K_i$, or satisfies $a_i^m=b$ for some natural number $m$ and $b\in K_i$.
\end{definition}

\begin{definition}
An algebraic function $z=z(x)$ of one variable is said to be representable in  $k$-radicals if the extension $K(z)/K$ is $k$-radical, where $K=F(x)$ is the field of rational functions over the base field $F$.
\end{definition}

In particular an algebraic function is representable in $1$-radicals if and only if it is representable in radicals.

In this paper we prove the following theorem:

\begin{theorem}
\label{thm:ksolvabilitynonprimitive}
A complex polynomial is invertible in $k$-radicals if and only if it is a composition of polynomials of degree at most $k$, power polynomials, Chebyshev polynmomials and polynomials from the following list (which depends on $k$):
\begin{enumerate}
\item for $1\leq k\leq 4$,  polynomials of degree 4,
\item \label{case:pgl25} for $k=5$, polynomials of degree 6 with monodromy group isomorphic to $PGL_2(5)$ with its natural action on the points of the projective line $P^1(F_5)$,
\item \label{case:pgammal29} for $k=6$, polynomials of degree 10 with monodromy group isomorphic to $P\Gamma L_2(9)$ with its natural action on the points of the projective line $P^1(F_{9})$,
\item \label{case:pgl27} for $k=7$, polynomials from list \ref{case:pgammal29} above and polynomials of degree 8 with monodromy group isomorphic to $PGL_2(7)$  with its natural action on the points of the projective line $P^1(F_7)$,
\item \label{case:pgl42} for $8\leq k\leq 14$, polynomials from list \ref{case:pgammal29} and polynomials of degree 15 with monodromy group isomorphic to $PSL_4(2)$  with its natural action either on points, or hyperplanes of the projective space $P^3(F_2)$.
\end{enumerate}

\end{theorem}

\begin{remark}
In particular for $k\geq 15$ a polynomial is invertible in $k$-radicals, if and only if it is a composition of power polynomials, Chebyshev polynomials and polynomials of degree at most $k$.
\end{remark}

\section{Ritt's theorem}

Theorem \ref{thm:ksolvabilitynonprimitive} on polynomials invertible in $k$-radicals can be considered as a generalization of Theorem \ref{thm:ritt} of Ritt on polynomials invertible in radicals. The outline of its proof is as follows:

\begin{enumerate}
\item \textit{Every polynomial is a composition of primitive ones:} Every polynomial is a composition of polynmomials that are not themselves compositions of polynomials of degree 2 and higher. Such polynomials are called primitive.
\item \textit{Reduction to the case of primitive polynomials:} It follows from the definition of being invertible in radicals that a composition of polynomials is invertible in radicals if and only if each polynomial in the composition is invertible in radicals. Indeed, if each of the polynomials in composition is invertible in radicals, then their composition also is. Conversely, if a polynomial $R$ appears in the presentation of a polynomial $P$ as a composition $P=Q\circ R\circ S$ and $P^{-1}$ is representable in radicals, then $R^{-1}=Q\circ P^{-1}\circ S$ is also representable in radicals. Thus it is enough to classify only the primitive polynmomials invertible in radicals.
\item \label{step:kSolvabilityAndGlois}\textit{Galois group is responsible for representability in radicals:} It follows from Galois theory that an algebraic equation over a field of characteristic zero is solvable in radicals if and only if its Galois group is solvable.
\item \textit{A polynomial is invertible in radicals if and only if its monodromy group is solvable:} A polynomial $p(x)$ is invertible in radicals if and only if the Galois group of the equation $p(x)=w$ over the field $k(w)$ is solvable. According to a result of Jordan, for $k=\mathbf{C}$ this group can be identified with the monodromy group of the function $p^{-1}(w)$.
\item \label{step:PrimitiveSolvableGroupsWithCycle}\textit{A result on solvable primitive permutation groups containing a full cycle:} It follows from what we said above that a primitive polynommial is invertible in radicals if and only if its monodromy group is solvable. Since the monodromy group acts primitively on the branches of inverse of the polynomial and contains a full cycle (corresponding to a loop around the point at infinity on the Riemann sphere), the following group-theoretical result of Ritt is useful for the classification of polynomials invertible in radicals:
\begin{theorem}
\label{thm:primitivesolvablegroupswithacycle}
Let $G$ be a primitive solvable group of permutations of a finite set $X$ which contains a full cycle. Then either $|X|=4$, or $|X|$ is a prime number $p$ and $X$ can be identified with the elements of the field $F_p$ so that the action of $G$ gets identified with the action of the subgroup of the affine group $AGL_1(p)=\{x\to ax+b|a\in (F_p)^*,b\in F_p\}$ that contains all the shifts $x\to x+b$.
\end{theorem}
\item \label{step:ExceptionalMonodromies}\textit{Monodromy groups of primitive polynomials invertible in radicals:} It can be shown by applying Riemann-Hurwitz formula that among the groups in Theorem \ref{thm:primitivesolvablegroupswithacycle} only the following groups can be realized as monodromy groups of polynomials: 1. $G\subset S(4)$, 2. Cyclic group $G=\{x\to x+b\}\subset AGL_1(p)$, 3. Dihedral group $G=\{x\to \pm x+b\}\subset AGL_1(p)$.
\item \label{step:DescriptionOfPolys}\textit{Primitive polynomials invertible in radicals:} It can be easily shown (see for instance \cite{Ritt22Radicals}, \cite{Khovanskii07Variations}, \cite{BurdaKhovanskii11Branching}) that the following result holds: 

\begin{theorem}
If the monodromy group of a polynomial is a subgroup of the group $\{x\to \pm x+b\}\subset AGL_1(p)$, then up to a linear change of variables the polynomial is either a power polynomials or a Chebyshev polynomial.
\end{theorem}

Thus the polynomials with monodromy groups 1-3 are respectively 1. Polynomials of degree four. 2. Power polynomials up to a linear change of variables. 3. Chebyshev polynomials up to a linear change of variables.

In each of these cases the fact that the polynomial is invertible in radicals follows from solvability of its monodromy group or from explicit formulas for its inverse (see for instance \cite{BurdaKhovanskii11Branching}).
\end{enumerate}

The outline of the proof of Theorem \ref{thm:ksolvabilitynonprimitive} is completely parallel to the outline discussed above. For the step \ref{step:kSolvabilityAndGlois} we use results from \cite{Khovanskii08Book}, for step \ref{step:PrimitiveSolvableGroupsWithCycle} --- results from \cite{Feit80Simple} and \cite{Jones02Cyclic}, for step \ref{step:ExceptionalMonodromies} --- results from \cite{Muller93Primitive}, \cite{Jones02Cyclic} and, finally, for step \ref{step:DescriptionOfPolys} --- results from \cite{JonesZvonkin02Cacti}, \cite{Adrianov97PlaneTrees}, \cite{CassouNogues99Factorizations} and personal communication with P. M\"{u}ller and A. Zvonkin.

\section{Background on representability in $k$-radicals}

It follows from the definition of a polynomial is invertible in $k$-radicals that a composition of polynomials is invertible in $k$-radicals if and only if each one of the polynomials in composition is invertible in $k$-radicals. Thus a polynomial is invertible in $k$-raddicals if and only if it is a composition of primitive polynomials invertible in $k$-radicals. In what follow we only consider primitive polynomials invertible in  $k$-radicals.

Invertibility of a polynomial in radicals depends only on its monodromy group:

\begin{definition}
A group $G$ is $[k]$-solvable if there exist subgroups $1=G_0\triangleleft G_1 \triangleleft \ldots \triangleleft G_{n-1} \triangleleft G_n=G$ such that for each $i>0$, $G_i/G_{i-1}$ is either abelian, or admits a faithful action on a set with $\leq k$ elements.
\end{definition}

It can be easily shown that this definition is equivalent to the following: 

\begin{definition}
A group $G$ is $[k]$-solvable if there exist subgroups $1=G_0\triangleleft G_1 \triangleleft \ldots \triangleleft G_{n-1} \triangleleft G_n=G$ such that for each $i>0$, $G_i/G_{i-1}$ is a simple group, which is either abelian, or contains a subgroup of index $\leq k$.
\end{definition}

The following result from \cite{Khovanskii08Book} describes when a field extension is $k$-radical:

\begin{theorem}
\label{thm:ksolvability}
An extension $L/K$ of fields of characteristic zero is $k$-radical if and only if the Galois group $Gal(L/K)$ is $[k]$-solvable.
\end{theorem}

In particular a polynomial is invertible in $k$-radicals if and only if its monodromy group is $[k]$-solvable.

\section{Results of Feit, M\"{u}ller and Jones}

The following result on primitive permutation groups containing a full cycle had been derived by Feit as a consequence of classification of finite simple groups \cite{Feit80Simple}. We provide a version of it due to Jones in  \cite{Jones02Cyclic}, in which the formulation of case \ref{case:JonesRefinement} is more accurate than the one in \cite{Feit80Simple}:

\begin{theorem}
A primitive group of permutations of $n$ elements contains a full cycle if and only if one of the following conditions holds:
\begin{enumerate}
\item $G = S_n$
\item $n$ is odd, $G=A_n$ is the group of even permutations acting naturally on $n$ elements,
\item  $n$ is prime, $C_n \subseteq G \subseteq AGL_1(n)$ acting naturally on the field $F_n$, where $C_n$ denotes a cyclic group of shifts inside the affine group $AGL_1(n)$.
\item \label{case:JonesRefinement} $n = \frac{q^d-1}{q-1}$, where $q$ is a power of prime and $PGL_d(q) \subseteq G \subseteq P\Gamma L_d(q)$ acting naturally either on points or on hyperplanes of the projective space $P(F_q^d)$, 
\item $n=11$ and $G = PSL_2(11)$ acting on $11$ cosets of one of two of its subgroups of index $11$,
\item $n=11$ and $G$ is Mathieu group $M_{11}$ acting naturally on $11$ elements,
\item $n=23$ and $G$ is Mathieu group $M_{23}$ acting naturally on $23$ elements.
\end{enumerate}
\end{theorem}

Using this result, Riemann-Hurwitz formula, M\"{u}ller proved the following result on monodromy groups of primitive polynomials \cite{Muller93Primitive}:

\begin{theorem}
\label{thm:mullerthm}
A group of permutations of $n$ elements is a monodromy group of a primitive polynomial if and only if one of the following conditions holds:
\begin{enumerate}
\item $G=S_n$
\item $n$ is odd, $G=A_n$ is the group of even permutations acting naturally on $n$ elements,
\item  $n$ is prime, $C_n \subseteq G \subseteq D_n=\{x\to \pm x + b \mod n\}$ acting naturally on the field $F_n$,
\item $n=11$ and $G = PSL_2(11)$ acting on $11$ cosets of one of its subgroups of index  $11$,
\item $n = \frac{p^d-1}{p-1}$, where $p$ is a prime number and $G=PGL_d(p)$ acting naturally either on points or on hyperplanes of the projective space $P(F_p^d)$, where $(p,d)$ is one of the following pairs: $(5,2),(7,2),(2,3),(3,3),(2,4),(2,5)$ (in these cases $n$ is, respectively 6,8,7,13,15,31) 
\item $n = \frac{q^d-1}{q-1}$, where $q$ is a power of a prime number and $G=P\Gamma L_d(q)$ acting naturally either on points or on hyperplanes of the projective space $P(F_q^d)$, where $(q,d)$ is one of the following pairs: $(8,2),(9,2),(4,3)$  (in these cases $n$ is, respectively 9,10,21) 
\item $n=11$ and $G$ is Mathieu group $M_{11}$ acting naturally on $11$ elements,
\item $n=23$ and $G$ is Mathieu group $M_{23}$ acting naturally on $23$ elements.
\end{enumerate}
\end{theorem}

\section{$[k]$-solvable monodromy groups of primitive polynomials}

According to Theorem \ref{thm:ksolvability} a polynomial is invertible in $k$-radicals if and only if its monodromy group is $[k]$-solvable, i.e. if its monodromy group $G$ contains subgroups $1=G_0\triangleleft G_1 \triangleleft \ldots \triangleleft G_{n-1} \triangleleft G_n=G$ such that for each $i>0$, $G_i/G_{i-1}$ is a simple group which is either abelian or contains a subgroup of index $\leq k$.

For each group from Theorem \ref{thm:mullerthm} we can determine the smallest $k$ for which it is $[k]$-solvable:

\begin{theorem}
Let $G$ be a group of permutations of $n$ elements, appearing in Theorem \ref{thm:mullerthm}. The group $G$ is $[k]$-solvable if and only if:
\begin{enumerate}
\item $k$ is any natural number and \\ $G=S_n,n\leq 4$ or \\
 $n$ is prime and $C_n \subseteq G \subseteq D_n=\{x\to \pm x + b \mod n\}$
\item $k\geq n$ and \\ $G=S_n$, or \\ $G=A_n$ for odd $n\geq 5$, or \\ $G = PSL_2(11)$ or $G=M_{11}$ for $n=11$, or \\ $G=M_{23}$ for $n=23$, or \\ $G=PGL_3(2)$ for $n=7$, or \\ $G=PGL_3(3)$ for $n=13$, or \\ $G=PGL_5(2)$ for $n=31$, or \\ $G=P\Gamma L_3(4)$ for $n=21$, or \\ $G=P\Gamma L_2(8)$ for $n=9$,
\item $G=PGL_2(5)$ , $k\geq 5$,
\item $G=P\Gamma L_2(9)$, $k\geq 6$,
\item $G=PGL_2(7)$, $k\geq 7$,
\item $G=PGL_4(2)$, $k\geq 8$.
\end{enumerate}
\end{theorem}

\begin{proof}

Let $G$ be a finite group and let $\{e\}=G_0 \triangleleft G_1 \triangleleft \ldots \triangleleft G_n=G$ be its composition series. Then the smallest $k$ for which $G$ is $[k]$-solvable is the smallest $k$ for which all the composition factors $G_{i+1}/G_i$ are either abelian or contain a proper subgroup of index at most $k$.

The group $A_n$, $n\geq 3$ doesn't contain a proper subgroup of index smaller than $n$ (otherwise $A_n$ can be embedded in $S_k$ for $k<n$ and $n!/2$ is $<k!$).

The group $A_n$ is a composition factor of groups $S_n$ and $A_n$, $n\geq 5$ from Theorem \ref{thm:mullerthm}, and hence these groups are $[k]$-solvable only for $k\geq n$.

The simple groups $M_{11}$ and $M_{23}$ don't have a proper subgroup of index smaller than 11 and 23 respectively \cite{ATLAS}, and thus they are $k$-solvable only for $k\geq 11$ and $k\geq 23$ respectively.

Compositional factors of groups $PGL_n(q)$ and $P\Gamma L_n(q)$ (for $n\geq 2$ and $q\neq 2,3$) are either abelian or isomorphic to the simple group $PSL_n(q)$, as can be seen from the natural homomorphisms onto abelian groups $P\Gamma L_n(q)\to \operatorname{Aut}(F_q)$ with kernel $PGL_n(q)$ and $PGL_n(q)\xrightarrow{\det} F_q^*/(F_q^*)^n$ with kernel $PSL_n(q)$ ($(F_q^*)^n$ is the subgroup of invertible elements of $F_q$ that are $n$-th powers). For small $n$ and $q$ the smallest index of a proper subgroup of $PSL_n(q)$ can be found in \cite{ATLAS} (we use the notation $L_n(q)$ for $PSL_n(q)$).

\begin{table}[h]
\centering
\begin{tabular}{|c|*{10}{@{\hspace{0.5ex}}c@{\hspace{0.5ex}}|}}
\hline $G$ & $L_2(5)$ & $L_2(7)$ & $L_3(2)$ & $L_2(11)$ & $L_2(8)$ & $L_2(9)$ & $L_3(3)$ & $L_4(2)$ & $L_3(4)$ & $L_5(2)$ \\ 
\hline $k$ & 5 & 7 & 7 & 11 & 9 & 6 & 13 & 8 & 21 & 31 \\ 
\hline 
\end{tabular} 
\end{table}

In all the cases except $L_2(5)$,$L_2(7)$,$L_2(9)$,$L_4(2)$ this $k$ coincides with the number of elements on which the corresponding group from Theorem \ref{thm:mullerthm} acts.

In cases $L_2(5)$,$L_2(7)$,$L_2(9)$,$L_4(2)$ one has the following exceptional isomorphisms: $PSL_2(F_5)=A_5$, $PSL_2(F_7)=PSL_3(F_2)$, $PSL_2(F_9)=A_6$, $PSL_4(F_2)=A_8$.
\end{proof}

A polynomial of prime degree with cyclic or dihedral monodromy group is, up to a linear change of variables, a power polynomial or Chebyshev polynomial respectively. Thus we obtain the following theorem:

\begin{theorem}
\label{thm:ksolvabilitygroup}
A primitive polynomial is invertible in $k$-radicals if and only if it has degree at most  $k$, or one of the following conditions holds:
\begin{enumerate}
\item $1\leq k$, the degree of the polynomial is a prime number and up to a linear change of variables the polynomial is a power polynomial or Chebyshev polynomial,
\item $k\leq 3$, the degree of the polynomial is 4,
\item \label{case_pol:pgl25} $k=5$, the degree of the polynomial is 6 and its monodromy group is $PGL_2(5)$,
\item \label{case_pol:pgammal29} $6 \leq k\leq 9$, the degree of the polynomial is 10 and its monodromy group is $P\Gamma L_2(9)$,
\item \label{case_pol:pgl27} $k=7$, the degree of the polynomial is 8 and its monodromy group is $PGL_2(7)$,
\item \label{case_pol:pgl42} $8\leq k\leq 14$, the degree of the polynomial is 15 and its monodromy group is $PGL_4(2)$.
\end{enumerate}

\end{theorem}

The polynmomials appearing in the exceptional cases \ref{case_pol:pgl25}-\ref{case_pol:pgl42} can be described explicitly: in cases \ref{case_pol:pgl25}-\ref{case_pol:pgl27} there is only a finite number of such polynomials up to a linear change of variables, while in case \ref{case_pol:pgl42} equivalence classes of such polynomials up to linear change of variables form two one-parametric families. Below we describe such polynomials.

\subsection{Polynomials, invertible in  5-radicals}
\label{section:deg5}

According to Theorem \ref{thm:ksolvabilitygroup}, polynomials invertible in 5-radicals are compositions of power polynomials, Chebyshev polynomials, polynomials of degree at most  5 and polynomials of degree 6 with monodromy group isomorphic to the group $PGL_2(5)$ with its natural action on the 6 points of the projective line over the field $F_5$ (the dual action of the group $PGL_2(5)$ on hyperplanes of the projective line over $F_5$ is the same as the action on points, since in this case hyperplanes are in fact just points).

\begin{theorem}
A primitive polynomial of degree six is invertible in 5-radicals if and only if one of the following conditions holds:
\begin{itemize}
\item The monodromy group of the polynomial is isomorphic to the group $PGL_2(5)$ with its natural action on $P^1(F_5)$
\item The passport of the polynomial is $[2^21^2,4^11^2]$
\item Dessin d'enfant of the polynomial is \\
\includegraphics[width=3cm]{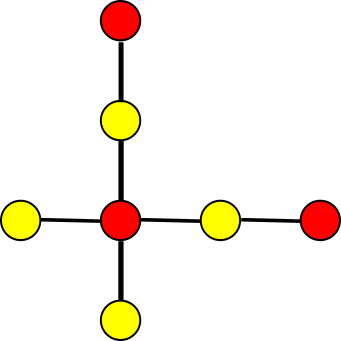}
\item By means of an affine change of variables the polynomial can be brought to the form $p(z)=z^4(z^2+6z+25)$
\end{itemize}
\end{theorem}

\begin{proof}
The permutations of 6 elements given by the action of $PGL_2(5)$ on $P^1(F_5)$ have cyclic structures $1^6,2^21^2,2^3,4^11^2,3^2,5^11^1,6^1$. Since the derivative of a polynomial of degree 6 has 5 roots counted with multiplicities, the cyclic structures of permutations corresponding to small loops around the critical values must be either $2^21^2,2^3$, or $2^21^2,4^11^2$. The first choice corresponds (according to Theorem 18 from \cite{Khovanskii07Variations}) to the case of Chebyshev polynomial. A polynomial with passport $[2^21^2,4^11^2]$ can have one of the two dessin d'enfants:\\
\includegraphics[width=4cm]{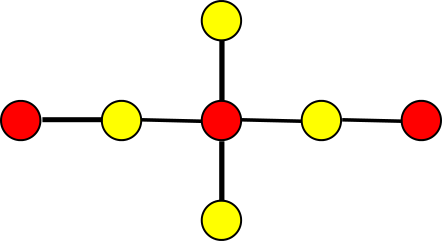} \hspace{9ex}
\includegraphics[width=3cm]{pgl25topdessin.png}

A polynomial with the first dessin d'enfant is a composition of a polynomial of degree 3 and a polynomial of degree 2.

The monodromy group of a polynomial with the second dessin d'enfant is $PGL_2(5)$.

Indeed, if one labels the edges of the dessin as in the picture below,\\
\includegraphics[width=4cm]{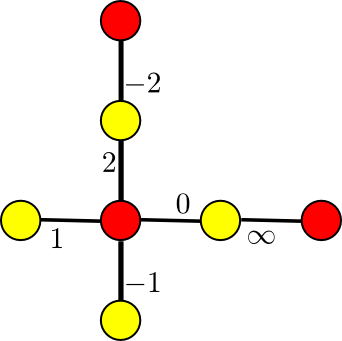}\\
then the small loops around the critical values correspond to the permutations $x\to \frac{1}{x} \mod 5$ and $x\to 2x+2 \mod 5$, which generate the group $PGL_2(5)$.

We now show that by an affine change of variables a polynomial of degree six with monodromy group $PGL_2(5)$ can be brought to the form $z^4(z^2+6z+25)$.

As we found above, the passport of such a polynomial $p$ is $[2^21^2,4^11^2]$.  By an affine change of coordinates one can make the point of multiplicity 4 to be at zero and make the polynomial vanish at this point. One can also make the leading coefficient of the polynomial be 1. Then the polynomial has the form $p(z)=z^4(z^2+az+b)$. Its derivative  is $p'(z)=z^3(6z^2+5az+4b)$. The values of the polynomial $p$ at the zeroes of the factor $6z^2+5az+4b$ must be equal, and hence the remainder of division of $p$ by $6z^2+5az+4b$ must be a constant polynomial. The coefficient at $z$ of the remainder of division of $p$ by $6z^2+5az+4b$ is $\frac{1}{6^5}a(96b-25a^2)(36b-25a^2)$. If $a=0$, then the polynomial $p$ is a composition of a polynomial of degree 3 and the polynomial $z^2$. If $96b=25a^2$, then $6z^2+5az+4b$ is a complete square, and hence the passport of $p$ is not $[2^21^2,4^11^2]$. Finally if $36b=25a^2$, then by a linear change of variables the polynomial $p$ one can make $p$ to be the polynomial $z^4(z^2+6z+25)$ with critical values $0$ and $-\frac{2^4 5^5}{3^3}$.

A picture of the dessin d'enfant of this polynomial on which the preimage of the upper half-plane is colored black (and red and yellow dots are the preimages of the critical values) is as follows:

\includegraphics[width=5cm]{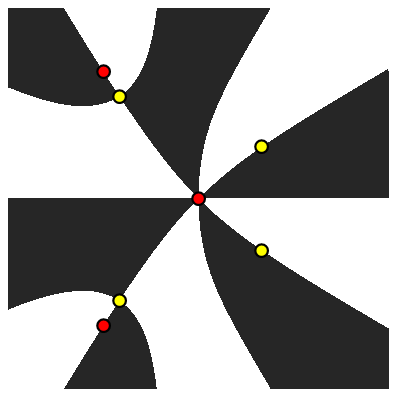}

\end{proof}

\subsection{Polynomials invertible in 6-radicals}
\label{sec:pgammal29}
According to Theorem \ref{thm:ksolvabilitygroup}, polynomials invertible in 6-radicals are compositions of power polynomials, Chebyshev polynomials, polynomials of degree at most 6 and polynomials of degree 10 with monodromy group isomorphic to the group $P\Gamma L_2(9)$ with its natural action on the 10 points of projective line over the field with nine elements $F_9$.

\begin{theorem}
A primitive polynomial of degree 10 is invertible in 6-radicals if and only if one of the following conditions holds:
\begin{itemize}
\item The monodromy group of the polynomial is isomorphic to the group $P\Gamma L_2(9)$ with its natural action on $P^1(F_9)$
\item The dessin d'enfant of the polynomial is \\
\includegraphics[width=3cm]{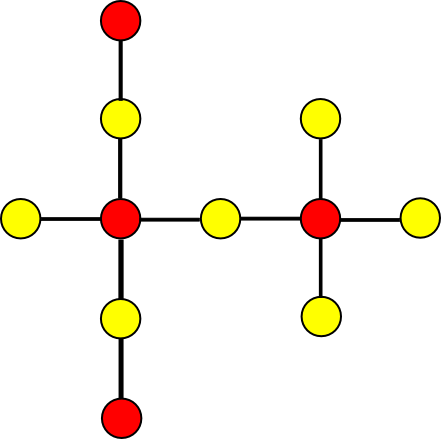}
\item By means of an affine change of variables the polynomial can be brought to the form $p(z)=\left(z^2-\frac{81}{500}\right)^4 \left(z^2+z+\frac{189}{500}\right)$
\end{itemize}
\end{theorem}

\begin{proof}
One can check \cite{Muller93Primitive}, p. 10, that the only possible passport of a polynomial of degree 10 with monodromy group $P\Gamma L_2(9)$ is the passport $[2^31^4,4^21^2]$.

We will let $i$ denote an element $i\in F_9$ satisfying $i^2=-1$. We will also denote the Frobenius automorphism of the field $F_9$ by $x\to \overline{x}$.

The group $P\Gamma L_2(9)$ acting on 10 elements of the projective line over the field $F_{9}$ contains only one conjugacy class of a 10-cycle: it is the class $C_1$ of the element $\frac{1+x}{i-x}$. It also contains only one conjugacy class $C_2$ of an element with cyclic structure $2^31^4$: it is the class of the element $x\to \overline{x}$. There are two conjugacy classes of elements with cyclic structure $1^24^2$: the class $C_3$ of element $x\to (1+i) \overline{x}$ and the class $\overline{C}_3$ of the element $x\to (1+i) x$. Only the class $C_3$ can correspond to local monodromy of our polynomial, since the product of elements of classes $C_1,\overline{C}_3$ belongs to the subgroup $PGL_2(9)$ of the group $P\Gamma L_2(9)$, and thus can't belong to $C_2$. One can verify that there exists only one solution (up to conjugacy) of the equation $\sigma_1\sigma_2\sigma_3=1$ with $\sigma_i\in C_i$: $\sigma_1=x\to\frac{1+x}{i-x}$,$\sigma_2=x\to \overline{x}$,$\sigma_3=x\to \frac{i\overline{x}-1}{\overline{x}+1}$. Thus the branching data for our polynomial are rigid \cite{Volklein96GroupsAsGalois}, Definition 2.15. Hence our polynomial is defined over the rationals \cite{Volklein96GroupsAsGalois}, Theorem 3.8.

It follows from the considerations above that the dessin d'enfant of the polynomial of degree 10 with monodromy group $P\Gamma L_2(9)$ is as follows:\\
\includegraphics[width=3cm]{pgammal29topdessin.png}

Conversely, the monodromy group of a polynomial with such dessin is isomorphic to $P\Gamma L_2(9)$, because one can label the edges of the dessin with elements of $P^1(F_9)$ as follows:\\
\includegraphics[width=4.5cm]{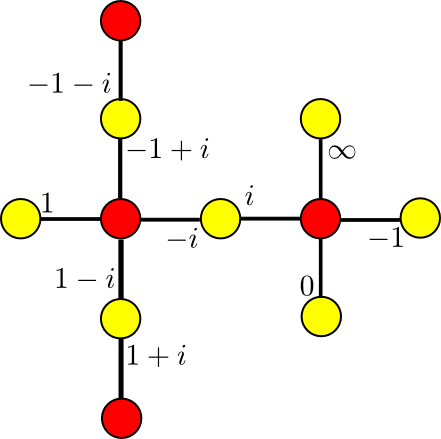}\\
Then local monodromies around the critical values correspond to the permutations $x\to \overline{x}$ and $x\to \frac{i\overline{x}-1}{\overline{x}+1}$, which generate the group $P\Gamma L_2(9)$.

We now show that using an affine change of variables the polynomial can be brought to the form $\left(z^2-\frac{81}{500}\right)^4 \left(z^2+z+\frac{189}{500}\right)$. By means of a change of variables defined over the rationals one can make sure that the critical value corresponding to the critical points of order 4 is zero. One can also make the average of the these two critical points be at zero. By means of a further change of variables one can bring the polynomial to the form $p(z)=(z^2-a)^3(z^2+z+b)$. In this case  $p'(z)=(z^2-a)^3(10z^3+9z^2+(8b-2a)z-a)$. Since the values of $p$ at the zeroes of the polynomial $q_3(z)=10z^3+9z^2+(8b-2a)z-a$ must be equal, the remainder from division of $p$ by $q_3$ must be a constant polynomial. Equating the coefficients at $z$ and $z^2$ of this remainder to zero and eliminating the variable $b$ we find that the value of $a$ can be equal either to $\frac{-27}{100}$, or to $\frac{81}{500}$, or to a root of a polynomial of degree 5 or 9, that is irreducible over the rationals.

The value $a=-\frac{27}{100}$ corresponds to the case when $q_3$ is a complete cube, in which the passport of the polynmial $p$ is not the one that we want.

The value $a=\frac{81}{500}$ corresponds to $b=\frac{189}{500}$. 

The cases when $a$ is a root of irreducible over $\mathbf{Q}$ polynomials of degree 5 or 9 correspond to polynomials with monodromy groups different from $P\Gamma L_2(9)$ (we have seen above that our polynomial is defined over $\mathbf{Q}$).

Thus by means of an affine change of variables the polynomial can be made equal to the polynomial $\left(z^2-\frac{81}{500}\right)^4 \left(z^2+z+\frac{189}{500}\right)$ with critical values 0 and $\frac{2^43^{12}}{5^{15}}$.

A picture of the dessin d'enfant of this polynomial on which the preimage of the upper half-plane is colored black (and red and yellow dots are the preimages of the critical values) is as follows:

\includegraphics[width=5cm]{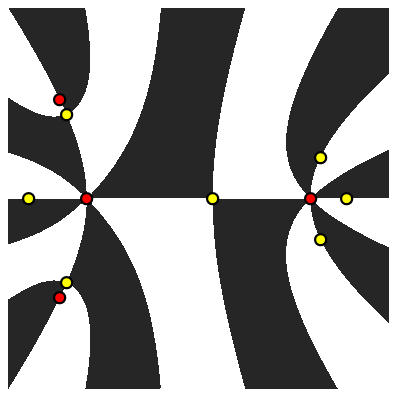}

\end{proof}

\subsection{Polynomials invertible in 7-radicals}

According to Theorem \ref{thm:ksolvabilitygroup}, polynomials invertible in 7-radicals are compositions of power polynomials, Chebyshev polynomials, polynomials of degree at most 7, polynomials of degree 10 with monodromy group isomorphic to $P\Gamma L_2(9)$ described in the section above, and polynomials of degree $8$ with monodromy group isomorphic to  $PGL_2(7)$ with its natural action on the 8 points of the projective line over the field $F_7$.

\begin{theorem}
A primitive polynomial of degree 8 is invertible in 7-radicals if and only if one of the following conditions holds:
\begin{itemize}
\item The monodromy group of the polynomials is isomorphic to the group $PGL_2(7)$ with its natural action on $P^1(F_7)$,
\item The dessin d'enfant of the polynomial is one of the following: \\
\includegraphics[height=3cm]{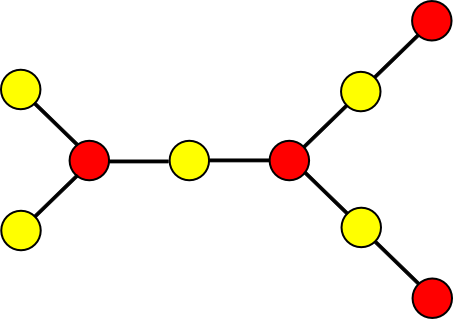}\hspace{8ex}
\includegraphics[height=3cm]{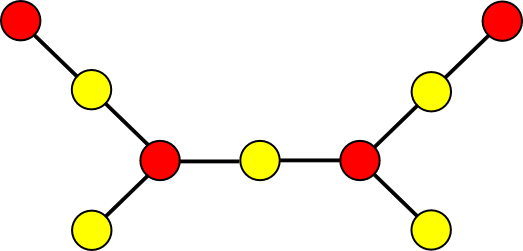}
\item By means of an affine change of variables the polynomial can be brought to the form  $p(z)=(z^2+\frac{25+ 22\sqrt{2}}{64})^3(z^2+z+\frac{97+ 54 \sqrt 2}{64})$ or to the form $p(z)=(z^2+\frac{25- 22\sqrt{2}}{64})^3(z^2+z+\frac{97- 54 \sqrt 2}{64})$.
\end{itemize}
\end{theorem}

\begin{proof}

One can verify \cite{Muller93Primitive}, p. 6, that the only possible passport of a polynomial of degree 8 woth moondromy group $PGL_2(7)$ is the passport $[2^31^2,3^21^2]$.

A polynomial with this passport can have one of the following dessins:

\begin{center}\includegraphics[height=3cm]{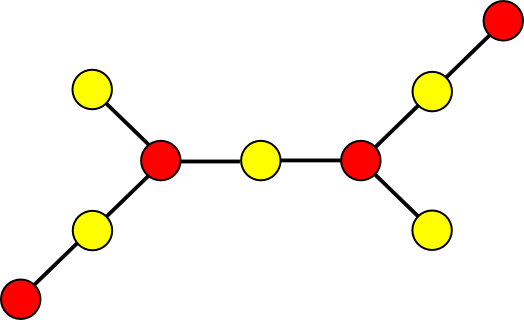}\end{center}
\includegraphics[height=3cm]{pgl27-1topdessin.png}\hspace{8ex}
\includegraphics[height=3cm]{pgl27-2topdessin.png}

A polynomial with the first dessin is a composition of polynomials of degree 2 and degree 4 (indeed, the dessin is invariant under rotation by 180 degrees).

The monodromy groups of polynomials with the last two dessins are isomorphic to the group  $PGL_2(7)$.
Indeed, if the edges of the dessins are labelled by elements of $P^1(F_7)$ as on the pictures below,\\
\includegraphics[height=3cm]{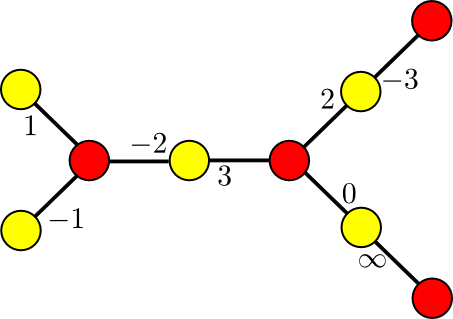}\hspace{8ex}
\includegraphics[height=3cm]{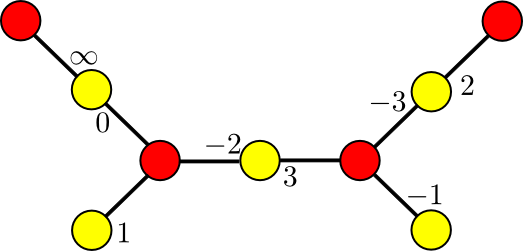}\\
then the local monodromies around the critical values correspond to the permutations $x\to \frac{1}{x}$ and $x\to 2-3x$ for the first of them, and to the permutations $x\to \frac{1}{x}$ and $x\to 1-3x$ for the second. In each case they generate the group  $PGL_2(7)$.

The group $PGL_2(7)$, acting on 8 elements of the projective line over $F_{7}$ contains two conjugacy classes of 8-cycles: the class $C_1$ of the element $\frac{3}{2-x}$ and the class $C_1'$ of the element $\frac{3}{1-x}$. It contains one conjugacy class $C_2$ of an element with cyclic structure $1^22^3$: it is the class of the element $x\to \frac{1}{x}$. There is one conjugacy class of an element with cyclic structure $1^23^2$: the class $C_3$ of the element $x\to 2x$. One can show that up to conjugacy there is only one solution of the equation $\sigma_1\sigma_2\sigma_3=1$ with $\sigma_i\in C_i$ (namely $\sigma_1=x\to \frac{3}{2-x}$, $\sigma_2=x\to \frac{1}{x}$, $\sigma_3=x\to 2-3x$). Also there is only one solution of the equation $\sigma_1\sigma_2\sigma_3=1$ with $\sigma_1\in C_1'$,$\sigma_2\in C_2$,$\sigma_3\in C_3$ (namely $\sigma_1=x\to \frac{3}{1-x}$, $\sigma_2=x\to \frac{1}{x}$, $\sigma_3=x\to 1-3x$).

Thus the branching data for our polynomial are rigid. The 8-cycle is defined over an extension of $\mathbf{Q}$ by a root of unity of order 8. Thus our polynomial is defined over the extension $\mathbf{Q}(\sqrt{2},i)$ of degree 4 over $\mathbf{Q}$.

As in the previous section, we can assume that the polynomial has the form  $p(z)=(z^2-a)^3(z^2+z+b)$. Then $p'(z)=(z^2-a)^2(8z^3+7z^2+(6b-2a)z-a)$. Since the values of the polynomial $p$ at the zeroes of the polynomial $q_3(z)=8z^3+7z^2+(6b-2a)z-a$ must be equal, the remainder from division of $p$ by $q_3$ must be a constant polynomial. Equating the coefficients at $z$ and $z^2$ of this remainder to zero and eliminating the variable $b$ we find that either $a=-\frac{343}{1728}$, or $4096 a^2+3200 a-343=0$, or $a$ is a root of a polynomial of degree 6 that is irreducible over the rationals.

The value $a=-\frac{343}{1728}$ corresponds to the case when $q_3$ is a complete cube. In this case the passport of the polynomial is not the one we are looking for.

The value $a=\frac{-25\pm 22\sqrt{2}}{64}$ corresponds to $b=\frac{97\mp 54 \sqrt 2}{64}$.

The case when $a$ is a root of an irreducible degree 6 polynomial over $\mathbf{Q}$ corresponds to polynomial with monodromy group different from $PGL_2(7)$ (our polynomial is defined over $\mathbf{Q}(\sqrt 2, i)$).

Pictures of the dessin d'enfants of these polynomials on which the preimage of the upper half-plane is colored black (and red and yellow dots are the preimages of the critical values) is as follows:

\includegraphics[width=5cm]{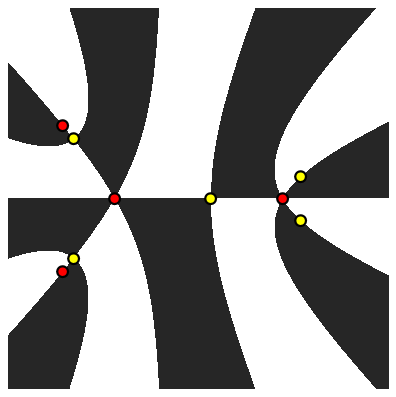}\hspace{6ex}
\includegraphics[width=5cm]{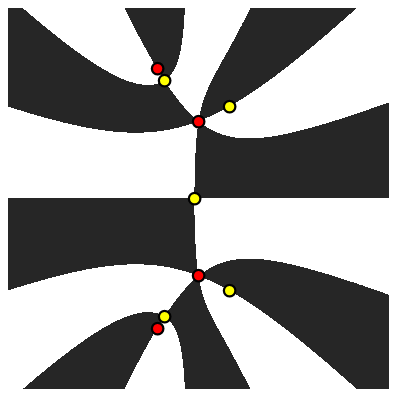}
\end{proof}

\subsection{Polynomials invertible in $k$-radicals, $8\leq k \leq 14$}
\label{section:deg15}

According to Theorem \ref{thm:ksolvabilitygroup}, polynomials invertible in $k$-radicals for $8\leq k \leq 14$ are compositions of power polynomials, Chebyshev polynomials, polynomials of degree at most $k$, polynomials of degree 10 with monodromy group  isomorphic to $P\Gamma L_2(9)$ described in section \ref{sec:pgammal29} and polynomials of degree $15$ with monodromy group isomorphic to $PGL_4(2)$ with its natural action either on the 15 points or on the 15 hyperplanes of the thee-dimensional projective space over the field $F_2$.

Polynomials of degree 15 with monodromy group isomorphic to $PGL_4(2)$ can have one of the following passports \cite{JonesZvonkin02Cacti},\cite{Adrianov97PlaneTrees}: $[2^61^3,2^41^7,2^41^7]$, $[4^32^11^1,2^41^7]$,$[4^22^21^3,2^61^3]$,$[6^13^22^11^1,2^41^7]$.

Such polynomials had been investigated in \cite{CassouNogues99Factorizations} in context of finding pairs of polynomials $g,h$ such that the curve $g(x)=h(y)$ is reducible. In \cite{CassouNogues99Factorizations} it is proved that a polynomial with monodromy group isomorphic to $PGL_4(2)$ and passport $[2^61^3,2^41^7,2^41^7]$ can be brought by an affine change of variables to the form
\begin{align*}g_t^a(x)&=
\frac{x^{15}}{15}
+(a-1) t x^{13}
+(a+7) t x^{12}
-(5 a+21) t^2 x^{11}
+2 (37 a-71) t^2 x^{10}\\
&-\frac{(261 a-349) (151598 t+141075 a-109260) t^2 }{3\cdot 151598}x^9\\
&-(649 a+703) t^3 x^8\\
&+\frac{3 (46 a+239) (76579 t+198260 a-462560) t^3}{76579} x^7\\
&-\frac{4 (548 a-1939) (259891 t+106365 a-26420) t^3}{259891} x^6\\
&+\frac{3 (1945 a-1581) (7278308 t+14685825 a-113700500) t^4}{5\cdot 7278308}  x^5\\
&+\frac{3 (3233 a+2051) (877444 t+1339725 a-2162500) t^4 }{877444} x^4\\
&+\frac{9 (9 a-133) \left(3\cdot 16816 t^2-162040 a t-320375a-1260960 t+23500\right) t^4 }{16816} x^3\\
&+\frac{9 (403 a-1559)  (2\cdot 2554 t+9165 a-39620) t^5}{2554}x^2\\
&-\frac{135}{16} (7 a+5) (4 t-75 a-100) (4 t+5 a-4)  t^5  x\\
&+675 (a-8) (t-16) t^6,
\end{align*} where $a$ is one of the two roots of the equation $a^2-a+4=0$ and $t$ is a complex number. 

Unfortunately the result is mentioned there only briefly and we haven't fully restored it. With the help of Alexandr Zvonkin we have come to the conclusion that most probably the following is true:

The polynomials from the families $g_t^a$ have monodromy group $PGL_4(2)$ with action on the points or on the hyperplanes of the space $P^3(F_2)$ depending on the choice of $a$ for all parameters $t\neq 0$.

All polynomials of degree 15 with monodromy group isomorphic to $PGL_4(2)$ can be brought by an affine change of variables to the form $g_t^a(x)$ for some $t$ and some choice of $a$. In particular the polynomials with monodromy group $PGL_4(2)$ with passports $[4^32^11^1,2^41^7]$,$[4^22^21^3,2^61^3]$,$[6^13^22^11^1,2^41^7]$ correspond to some values of the parameter. For instance for $t=75/4$ the polynomial $g_t^a$ has passport $[4^22^21^3,2^61^3]$ and dessin d'enfant

\includegraphics[height=5cm]{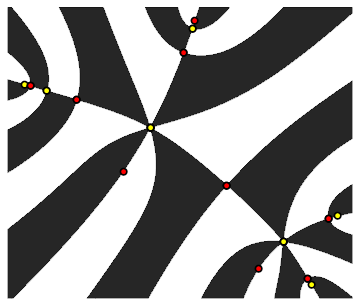}

(or its reflection for the other choice of $a$).

We will provide more details on the properties of the families $g_t^a$ and on how they can be found in a later version of this preprint.

Alexandr Zvonkin has kindly agreed to help us with this task.



\subsection{Polynomials invertible in $k$-radicals for $k \geq 15$}
\label{section:highdegree}

According to Theorem \ref{thm:ksolvabilitygroup}, polynomials invertible in $k$-radicals for $k \geq 15$ are compositions of power polynomials, Chebyshev polynomials and polynomials of degree at most $k$.

Thus there are no ``exceptional'' polynomials invertible in $k$-radicals for $k\geq 15$.

\bibliography{ksolvability}

\bibliographystyle{alpha}

\end{document}